\documentclass[12pt,journal]{IEEEtran}
\onecolumn
\usepackage{mathrsfs}
\usepackage{algorithmic}
\usepackage{algorithm}
\usepackage{amsmath}
\usepackage{amsthm}
\usepackage{graphicx}
\usepackage{amssymb}
\usepackage{epstopdf}
\usepackage{enumerate}
\usepackage{longtable,tabularx,float}
\usepackage{cite}
\usepackage{mathcomp}
\usepackage{multirow}
\usepackage{supertabular}
\usepackage{stmaryrd}
\usepackage{xcolor}
\usepackage{url}
\usepackage{makecell}
\usepackage[OT2,OT1]{fontenc}
\usepackage{bm}
\usepackage{bbm}
\usepackage{caption}
\newtheorem{corollary}{Corollary}[section]

\newtheorem{lemma}{Lemma}[section]

\newtheorem{theorem}{Theorem}[section]
\newenvironment{thmbis}[1]
 {\renewcommand{\thetheorem}{\ref{#1}}%
 \addtocounter{theorem}{-1}%
 \begin{theorem}}
 {\end{theorem}}
\newtheorem{problem}{Problem}[section]

\newtheorem{construction}{Construction}[section]

\newtheorem{remark}{Remark}[section]
\newtheorem{conjecture}{Conjecture}[section]

\newtheorem{question}{Question}[section]

\begin{document}

\title{On set systems with strongly restricted intersections}
\author{Xin~Wei, Xiande~Zhang, and~Gennian~Ge
 \thanks{X. Wei ({\tt weixinma@mail.ustc.edu.cn}) is with the School of Mathematical Sciences, University of Science and Technology of China, Hefei, 230026, Anhui, China.}
 \thanks{X. Zhang ({\tt drzhangx@ustc.edu.cn}) is with the School of Mathematical Sciences, University of Science and Technology of China, Hefei, 230026, Anhui, China, and  Hefei National Laboratory, Hefei, 230088, China.}
 \thanks{G. Ge ({\tt gnge@zju.edu.cn}) is with the School of Mathematical Sciences, Capital Normal University, Beijing, 100048, China.
}

}
\maketitle

\begin{abstract}
Set systems with strongly restricted intersections, called  $\alpha$-intersecting families for a vector $\alpha$, were introduced recently as a generalization of several well-studied intersecting families including the classical oddtown and eventown. Given a binary vector $\alpha=(a_1, \ldots, a_k)$, a collection $\mathcal F$ of subsets over an $n$ element set is an $\alpha$-intersecting family modulo $2$ if for each $i=1,2,\ldots,k$, all $i$-wise intersections of distinct members in $\mathcal F$ have sizes with the same parity as $a_i$. Let $f_\alpha(n)$ denote the maximum size of such a family. In this paper, we study the asymptotic behavior of $f_\alpha(n)$ when $n$ goes to infinity. We show that if $t$ is the maximum integer such that $a_t=1$ and $2t\leq k$, then $f_\alpha(n)\sim (t! n)^{1\slash t}$. More importantly, we show that for any constant $c$, as the length $k$ goes larger, $f_\alpha(n)$ is upper bounded by $O (n^c)$ for almost all $\alpha$. Equivalently, no matter what $k$ is, there are only finitely many $\alpha$ satisfying $f_\alpha(n)=\Omega (n^c)$.
 This answers an open problem raised by Johnston and O'Neill in 2023. All of our results can be generalized to modulo $p$ setting for any prime $p$ smoothly.

\end{abstract}


\section{Introduction}

The classical oddtown and eventown problems concern maximizing the number of subsets of a finite set subject to certain parity-related constraints. Given a collection $\mathcal F$ of subsets of an $n$ element set, we say $\mathcal F$ follows eventown (resp., oddtown) rules if all its members have even (resp., odd) size, and the intersection of any two different members in $\mathcal F$ has even size. Berlekamp~\cite{berlekamp1969subsets}  and Graver~\cite{graver1975boolean} independently proved that the maximum sizes of such families satisfying eventown rules and oddtown rules are $2^{\lfloor n\slash2 \rfloor}$ and $n$, respectively, which are the best possible. Their methods highlight the linear algebra method~\cite{babai1988linear} in extremal combinatorics.

There have been a lot of generalizations and applications for oddtown and eventown problems during the decades~\cite{deza1983functions, vu1997extremal, vu1999extremal, o2021towards, Wei2023}. As oddtown and eventown problems only consider the sizes of intersections of at most two sets, one direction of generalizations is to consider the size of $k$-wise intersections for big $k$. Very recently, Johnston and O'Neill \cite{Johnston2023} proposed the following general definition of $\alpha$-intersecting family with $\alpha$ being a vector.

 Given a finite set $X$, let  $2^X$ denote the collection of all subsets of $X$. Suppose $\mathcal F\subset 2^{[n]}$ with $[n]:=\{1,2,\ldots,n\}$. Let $\alpha :=(a_1, \ldots, a_k)\in \mathbb F_2^k$ being a vector of length $k$. Then $\mathcal F$ is called an {\it intersecting family of pattern $\alpha$ modulo $2$}, or {\it an $\alpha$-intersecting family modulo $2$}, if for each $\ell\in [k]$, the intersection of any $\ell$ distinct $F_1, \ldots, F_\ell\in\mathcal F$  is $a_\ell$ modulo $2$, that is,  $|F_1\cap\cdots\cap F_\ell|\equiv a_\ell \pmod 2$. We  omit the term ``modulo $2$'' for simplicity if there is no confusion.
  Let $f_\alpha(n)$ denote the maximum size of an $\alpha$-intersecting family $\mathcal F\subset 2^{[n]}$. For a given pattern $\alpha \in \mathbb F_2^k$, the central problem is to determine the value of $f_\alpha(n)$ for any $n$, or to estimate $f_\alpha(n)$ when $n$ goes to infinity.

Under the above notations, the study of classical oddtown and eventown problems shows that $f_\alpha(n)=2^{\lfloor n\slash2 \rfloor}$ when $\alpha= (0, 0)$ for eventown, and $f_\alpha(n)=n$ when $\alpha=(1, 0)$ for oddtown. The exact values of $f_\alpha(n)$ for any $\alpha$ of length $2$ or $3$, and the orders of $n$ in $f_\alpha(n)$ for any $\alpha$ of length $4$ are totally determined in \cite{Johnston2023}. It is worth noticing that some previously studied generalizations of oddtown and eventown problems are indeed $\alpha$-intersecting families for some specific $\alpha$. For example, Sudakov and Vieira \cite{sudakov2018two} introduced  the {\it strong $k$-eventown problem} for proving a nice stability result of {\it $k$-wise eventown},  where a strong $k$-eventown family is just an $\alpha$-intersecting family modulo $2$ with $\alpha=\bm 0\in \mathbb F_2^k$. The $k$-wise eventown problem requires the size of $k$-wise intersection to be even but no restriction on the $i$-wise intersections for $i<k$. Recently, O'Neill and Verstra\"{e}te~\cite{o2022note} considered the so-called {\it $(k, t+1)$-oddtown problem}, which is just to determine $f_\alpha(n)$ for $\alpha\in \mathbb F_2^k$ with the first $t$ coordinates of $\alpha$ being one and the remaining being zero.


In this paper, we focus on the estimation of $f_\alpha(n)$ for any given $k$ and $\alpha\in \mathbb F_2^k$ when $n$ goes to infinity. We need the following asymptotic notations to state our results. For functions $f, g: \mathbb N\to \mathbb R^+$, we write $f=o(g)$ if $\lim_{n\to\infty} f(n)\slash g(n)=0$, and $f=O(g)$ if there exists a constant $c>0$ such that $f(n)\le cg(n)$ for all $n\in\mathbb N$. If $f=O(g)$ and $g=O(f)$, we say $f=\Theta(g)$. If $g=O(f)$, we say $f=\Omega(g)$. Especially, if $\lim_{n\to\infty}f(n)\slash g(n)=1$, we say $f\sim g$.

There are some known results of $f_\alpha(n)$ for  certain patterns $\alpha$ of any given length $k$.  For the strong $k$-eventown problem, that is $\alpha=\bm 0\in \mathbb F_2^k$, we know that $f_{\bm 0}(n)\sim 2^{\lfloor n\slash 2\rfloor}$. The upper bound follows from the bound of $k$-wise eventown by Vu~\cite{vu1997extremal}, while the lower bound is given by the classical construction: let $B_1, \ldots, B_{\lfloor n\slash 2\rfloor}$ be disjoint subsets of $[n]$ of size $2$, then the family $\{\bigcup_{i\in S}B_i: S\subseteq[\lfloor n\slash 2\rfloor]\}$ is a $\bm 0$-intersecting family of size $2^{\lfloor n\slash 2\rfloor}$. For $\alpha=\bm 1\in \mathbb F_2^k$, that is the all-one vector, it is easy to get $f_{\bm 1}(n)= \Theta(2^{\lfloor n\slash 2\rfloor})$. To see this, if $\mathcal F\subset 2^{[n]}$ is a $\bm 1$-intersecting family, then the family $\{F\cup\{n+1\}: F\in\mathcal F\} \subset 2^{[n+1]}$ is a $\bm 0$-intersecting family. Hence $f_{\bm 1}(n)\le f_{\bm 0}(n+1)= O(2^{\lfloor n\slash 2\rfloor})$. Similarly, if $\mathcal F'\subset 2^{[n-1]}$ is a $\bm 0$-intersecting family, then $\{F'\cup\{n\}: F'\in\mathcal F'\} \subset 2^{[n]}$ is a $\bm 1$-intersecting family. Hence $f_{\bm 1}(n)\ge f_{\bm 0}(n-1)= \Omega(2^{\lfloor n\slash 2\rfloor})$.

For the $(k, t+1)$-oddtown problem, that is $\alpha$-intersecting with $\alpha=(1,1,\ldots, 1, 0,\ldots, 0)\in \mathbb F_2^k$, the first $t$ coordinates being one,  O'Neill and Verstra\"{e}te~\cite{o2022note} proved that $f_\alpha(n)=\Theta(n^{1\slash t})$  whenever $k\ge 2t$. Our first contribution is to improve this result for more patterns.

\begin{theorem}\label{theorem_extend_k_t_oddtown} Let $k\geq 2t$ be positive integers. For any  $\alpha\in \mathbb F_2^k$ with $t$ being the maximum integer such that $a_t=1$, we have
$$f_\alpha(n)\sim (t! n)^{1\slash t}.$$
\end{theorem}

Note that Theorem~\ref{theorem_extend_k_t_oddtown} gives an accurate estimation on  the value of $f_\alpha(n)$ for a fraction of about $\frac{1}{2^{k/2}}$ of the patterns $\alpha\in \mathbb F_2^k$, including the $(k, t+1)$-oddtown problem.  In \cite{Johnston2023}, the authors noticed that for any length $k\ge 3$, there are only eight patterns $\alpha\in \mathbb F_2^k$  satisfying $f_\alpha(n)=\Omega(n)$, which are big, while all other patterns satisfy $f_\alpha(n)=O(\sqrt n)$. They wondered whether there is also a constant upper bound on the number of $\alpha$ satisfying $f_\alpha(n)=\Omega(\sqrt n)$ for any $k$, which led them to raise the following  question.


\begin{question}[\hspace{-0.01em}\cite{Johnston2023}]\label{question_1}
Does there exist an absolute constant $C$ such that, for all $k\geq 2$, there are at most $C$ vectors $\alpha\in\mathbb F_2^k$ satisfying $f_\alpha(n)=\Omega(\sqrt n)$?
\end{question}

  Our second result, which is also our main result, gives an affirmative answer to Question~\ref{question_1}. In fact, our result answers a general question for $f_\alpha(n)=\Omega(n^{1/t})$ for infinitely many integers $t\geq 1$. See below.

\begin{theorem}\label{theorem_main}
For any positive integers $\ell$ and $k\geq 2$, the number of patterns $\alpha\in\mathbb F_2^k$ such that $f_{\alpha}(n)=\Omega(n^{1\slash (2^\ell -1)})$ is no more than $2^{2^{\ell+1}-1}$. For all other patterns not satisfying this property, $f_{\alpha}(n)=O(n^{1\slash 2^\ell})$.
\end{theorem}


 Applying Theorem~\ref{theorem_main} with $\ell=2$, the number of patterns $\alpha\in\mathbb F_2^k$ such that $f_{\alpha}(n)=\Omega(n^{1\slash 3})$ is no more than $2^7$.
This answers Question~\ref{question_1} with $C=2^7$. Obviously, $C=2^7$ is not sharp, because by applying Theorem~\ref{theorem_extend_k_t_oddtown} there are always some pattern $\alpha\in \mathbb F_2^k$ with $f_\alpha(n)=\Theta(n^{1\slash 3})$ when $k\ge 6$. We do not take an effort to find the best $C$ here, but leave it as an open problem. Further, we can conclude that the number of patterns $\alpha\in\mathbb F_2^k$ such that $f_{\alpha}(n)=\Omega(n^{1\slash (2^\ell -1)})$ is exactly $2^{2^{\ell+1}-1}$ when $k\ge 2^{\ell+1}$ (see Remark~\ref{rmk_equal}).

In the process of proving Theorem~\ref{theorem_extend_k_t_oddtown} and  Theorem~\ref{theorem_main}, we partially confirm a conjecture posed in \cite{o2022note} as a byproduct.
\begin{conjecture}[\hspace{-0.01em}\cite{o2022note}, Conjecture 1]\label{conjecture_1}
Let $t, k$ be integers with $t\ge 2$ and $2t-2>k$. If $(\mathcal A_1, \ldots, \mathcal A_k)$ are set families of an $n$ element set with $\mathcal A_j=\{A_{j, i}: i\in [m]\}$ where $|\bigcap_{j=1}^k A_{j, i_j}|$ is even if and only if at least $t$ of those $i_j$ are distinct, then $m=O(n^{1\slash\lfloor k\slash 2\rfloor})$.
\end{conjecture}
We show the correctness of Conjecture~\ref{conjecture_1}  when $\mathcal A_1=\cdots=\mathcal A_k$ and $k$ is a  power of $2$. See Corollary~\ref{cor-conje}. Note that in this special case, it is indeed to prove that $f_\alpha(n)=O(n^{1\slash\lfloor k\slash 2\rfloor})$ for $\alpha=(1,\ldots,1,0,\ldots,0)$ with $t>k/2$ ones, which is not involved in Theorem~\ref{theorem_extend_k_t_oddtown}.

Finally, we mention that the methodology for proving Theorem~\ref{theorem_extend_k_t_oddtown} and  Theorem~\ref{theorem_main} can be easily extended to $\alpha$-intersecting problems modulo $p$ for any fixed prime $p$. For a vector $\alpha=(a_1, \ldots, a_k)\in\mathbb F_p^k$,
 it is natural to define an {\it $\alpha$-intersecting family modulo $p$} as a collection $\mathcal F$ of subsets such that all distinct $i$-wise intersections of $\mathcal F$ have size $a_i$ modulo $p$.
 Let $f_\alpha(p, n)$ be the maximum size of an $\alpha$-intersecting family  $\mathcal F\subset 2^{[n]}$ modulo $p$.
 Following the notation in \cite{Johnston2023},  we write $a_i=\star$ to indicate that all distinct $i$-wise intersections have size non-zero modulo $p$. So in general, under the modulo $p$ setting with $p\ge 3$, one can consider the $\alpha$-intersecting problem for any $\alpha$ chosen in $\left\{\mathbb F_p\cup\{\star\}\right\}^k$. In  \cite{Johnston2023}, the values of $f_{\alpha}(3, n)$ for any pattern $\alpha\in \{0, \star\}^2\cup \{0, \star\}^3$ are studied carefully.

We list our results for the modulo $p$ setting here. For details of proofs,  we refer readers to the Appendix. The numberings of theorems are the same as in the Appendix.

\begin{thmbis}{theorem_k_reachable_2} Let $p$ be a prime.
 For any positive $t$ and  pattern $\alpha=(a_1, \ldots, a_k)\in\mathbb F_p^{k}$ with $k\ge 2t$, if
  $a_t\ne a_{t+1}=\cdots =a_k$,
we have $f_\alpha(p, n) =\Theta(n^{1\slash t})$.
Moreover, if $a_t-a_{t+1}=1$, $f_\alpha(p, n) \sim (t! n)^{1\slash t}$.
\end{thmbis}

\begin{thmbis}{theorem_main_2}
Let $p$ be a prime. For any positive integers $k$ and $\ell$, the number of patterns $\alpha\in\mathbb F_p^k$ such that $f_{\alpha}(p, n)=\Omega(n^{1\slash (p^\ell -1)})$ is no more than $p^{2\cdot p^\ell -1}$. For all other patterns not satisfying this property, $f_{\alpha}(p, n)=O(n^{1\slash p^\ell})$.
\end{thmbis}

Similarly, the number of patterns $\alpha\in\mathbb F_p^k$ such that $f_{\alpha}(p, n)=\Omega(n^{1\slash (p^\ell -1)})$ is exactly $p^{2\cdot p^\ell -1}$ when $k\ge 2\cdot p^\ell$.

\begin{thmbis}{theorem_k_reachable_2_star}
Let $p$ be a prime. For any  positive $\ell$ and any pattern $\alpha=(a_1, \ldots, a_k)\in \left\{\mathbb F_p\cup\{\star\}\right\}^k$ with $k\ge 2\ell$ such that $a_\ell\ne 0$ and $a_{\ell+1}=\cdots= a_k=0$,
we have $f_\alpha(p, n) = \Theta( n^{1\slash \ell})$. Moreover, if $a_\ell=1$ or $a_\ell=\star$, $f_\alpha(p, n) \sim (\ell! n)^{1\slash \ell}$.
\end{thmbis}

Our paper is organized as follows. In Section~\ref{sec:two-fam}, we introduce two basic constructions of families used throughout our paper and discuss some of their properties. In Section~\ref{sec_preliminary}, three lemmas from \cite{Johnston2023} are introduced and generalized to asymptotic settings. Section~\ref{sec_generalizations} is devoted to the proof of Theorem~\ref{theorem_extend_k_t_oddtown}. Our main result Theorem~\ref{theorem_main} is proven in Section~\ref{sec_main_result}. Some extensions of our results to modulo $p$ settings are given in the Appendix, and we conclude our paper in Section~\ref{sec_conclude}.

Notations: denote $[m, n]=\{m, m+1, \ldots, n\}$ for any integers $m\leq n$. Denote $\epsilon_s$ the $0$-$1$ vector with only one symbol $1$ in the $s$-th position, and denote $\dot{\epsilon}_{s,\ell}$ as the $0$-$1$ vector with symbol $1$ only in the $s$-th, $(s+\ell)$-th, $(s+2\ell)$-th, $\ldots$, positions, until the vector length is exceeded. For  vectors $\alpha_1, \ldots, \alpha_t\in \mathbb F_p^k$, denote $\langle\alpha_1, \ldots, \alpha_t\rangle$ the set consisting of all vectors in the form $\sum_{i=1}^t r_i\cdot \alpha_i$, where $r_i\in \mathbb F_p$, $i\in [t]$. As the function $f_\alpha(n)$ needs $n$ to be an integer, sometimes we omit the flooring and ceiling functions in the expressions of $n$ for brevity when there is no confusion.

\section{Two Constructions}\label{sec:two-fam}

In this section, we give two constructions of $\alpha$-intersecting families for certain pattern $\alpha$ modulo $2$. These two families are quite useful and will be utilized throughout. They also provide lower bounds on the size of intersecting families of certain pattern.

In our construction, we denote $\varphi(s)$ as the number of $2$-factors in $s!$ for any positive integer $s$. That is, $\varphi(s)$ is the maximum integer such that $2^{\varphi(s)}\mid s!$.

\begin{construction}\label{construction_building_block}
  For any integers $s$, $n$ and $i\in [0, 2^{\varphi(s)+1}-1]$, let $m$ be the largest integer satisfying $\binom m s\le n$ and $m\equiv i\mod 2^{\varphi(s)+1}$. Let $X$ be the set of all $s$-subsets of $[m]$. Now we construct two families $\mathcal F,\overline{\mathcal F}\subset 2^X$ for the given $s,i$ and $n$.

  \begin{itemize}
    \item[(1)] For each $j\in [m]$, let $E_j=\{e\in X: j\in e\}$, and let $\mathcal F_{s, i}(n)= \{E_1, \ldots, E_m\}$.
    \item[(2)] For each $j\in [m]$, let $\overline{E}_j=\{e\in X: j\notin e\}$, and let $\overline{\mathcal F}_{s, i}(n)= \{\overline{E}_1, \ldots, \overline{E}_m\}$.
  \end{itemize}

Since $|X|=\binom m s\le n$, we can fix an injection from $X$ to $[n]$. So $\mathcal F$ and $\overline{\mathcal F}$ can be viewed as families over $[n]$ in a natural way.

\end{construction}

One can see that  the family $\mathcal F_{s, i}(n)$ in Construction~\ref{construction_building_block} is a generalization of \cite[Construction 3]{Johnston2023}, where only the case $s=2$ and  $i=1$ was considered. Note that for any $i$, the size of $\mathcal F_{s, i}(n)$ and $\overline{\mathcal F}_{s, i}(n)$ is $m$, which satisfies $m \sim (s! n)^{1\slash s}$ when $n$ goes to infinity since $\binom m s\le n < \binom {m+2^{\varphi(s)+1}} s$ . 

Next, we show that these two families satisfy good intersection patterns.

\begin{lemma}\label{lemma_first_class_good}
For any positive integers $s$ and $n$, if $|\mathcal F_{s, s-1}(n)|=m\ge s+1$, then $\mathcal F_{s, s-1}(n)$ is $\epsilon_s$-intersecting modulo $2$ with $\epsilon_s\in \mathbb F_2^m$.
%
%
\end{lemma}
\begin{proof}
For any $\ell$ different members $E_{j_1}, \ldots, E_{j_\ell}$ in $\mathcal F_{s, s-1}(n)$, by definition,
$$\begin{aligned}
|\bigcap_{x\in [\ell]}E_{j_x}|&=|\{e\in X: \{j_1, \ldots, j_\ell\}\subset e\}|\\
 &=\binom{m-\ell}{s-\ell}.
\end{aligned}$$
If $\ell>s$, $|\bigcap_{x\in [\ell]}E_{j_x}|=\binom{m-\ell}{s-\ell}=0$; if  $\ell=s$, $|\bigcap_{x\in [\ell]}E_{j_x}|=\binom{m-s}{0}=1$.
For any $1\le \ell< s$, $$|\bigcap_{x\in [\ell]}E_{j_x}|=\frac{(m-\ell)(m-\ell-1)\cdots (m-s+1)}{(s-\ell)!}.$$
Consider the number of $2$-factors in the numerator and the denominator separately. Since $(s-\ell)!\mid s!$, the number of $2$-factors in the denominator is at most $\varphi (s)$. On the other hand, by our choice of $m$ in $\mathcal F_{s, s-1}(n)$, $m\equiv s-1 \mod 2^{\varphi(s)+1}$. Hence $2^{\varphi(s)+1}$ divides the numerator and $|\bigcap_{x\in [\ell]}E_{j_x}|$ is even.
\end{proof}

Note that $\epsilon_s, s\in [k]$ form a basis of $\mathbb F_2^k$ for any $k$, which will be quite useful in constructing large families satisfying any certain pattern $\alpha\in \mathbb F_2^k$. For example, see the proof of Lemma~\ref{theorem_lower_infty} in the next section. As a corollary, we have the following lower bound.
{
\begin{corollary}\label{cor_first_class_good}
 For any positive integers $s\leq k$,  $f_{\epsilon_s}(n)\geq (1-o(1))(s! n)^{1\slash s}$  with $\epsilon_s\in \mathbb F_2^k$.
\end{corollary}}

\begin{lemma}\label{lemma_second_class_good}
For any positive integer $s$ and $i\in [0, 2^{s}-1]$, if $|\overline{\mathcal F}_{2^s-1, i}(n)|=m\geq2^s$, then the family $\overline{\mathcal F}_{2^s-1, i}(n)$ is
$\dot{\epsilon}_{i+1, 2^s}$-intersecting modulo $2$ with $\dot{\epsilon}_{i+1, 2^s}\in\mathbb F_2^m$.

\end{lemma}
\begin{proof}
For any $\ell$ distinct members $\overline{E}_{j_1}, \ldots, \overline{E}_{j_\ell}$ in $\overline{\mathcal F}_{2^s-1, i}(n)$, by definition,
$$\begin{aligned}
|\bigcap_{x\in [\ell]}\overline{E}_{j_x}|&=|\{e\in X: \{j_1, \ldots, j_\ell\}{\cap e=\emptyset}\}|\\
 &=\binom{m-\ell}{2^s-1}.
\end{aligned}$$
Write $\binom{m-\ell}{2^s-1}$ in the form of $\left(\prod_{j\in [2^s-1]}(m-\ell-2^s+1+j)\right)\slash (2^s-1)!$.
 From the requirement of $m$ of $\overline{\mathcal F}_{2^s-1, i}(n)$, $m\equiv i\mod 2^{\varphi(2^s-1)+1}$ and hence $m\equiv i\mod 2^{s}$.

 When $\ell\equiv i+1\mod 2^s$, we have  $m-\ell\equiv -1 \mod 2^s$.
Under this relationship, for any $j\in [2^s-1]$, $m-\ell-2^s+1+j\equiv j\mod 2^s$, which means $m-\ell-2^s+1+j$ has the same number of $2$-factors as $j$. So $\prod_{j\in [2^s-1]}(m-\ell-2^s+1+j)$ has the same number of $2$-factors as $(2^s-1)!$. As a consequence, $|\bigcap_{x\in [\ell]}\overline{E}_{j_x}|=\binom{m-\ell}{2^s-1}$ is odd.

When $\ell\not\equiv i+1\mod 2^s$, there always exists some $j_0\in [2^s-1]$ such that $2^s\mid (m-\ell-2^s+1+j_0)$. We can split the denominator $(2^s-1)!$ into  three parts, that is,
$$(2^s-1)!=(2^s-1-j_0)!\cdot (2^s-j_0)\cdot \prod_{j\in [j_0-1]}(2^s+ j-j_0).$$
For any $j\in [j_0+1, 2^s-1]$, the number $(m-\ell-2^s+1+j)$ has the same number of 2-factors as $j-j_0$. For any $j\in [j_0-1]$, the number $(m-\ell-2^s+1+j)$ has the same number of 2-factors as $2^s+ j-j_0$. However, $(2^s-j_0)$ has less 2-factors than $(m-\ell-2^s+1+j_0)$. As a consequence, $|\bigcap_{x\in [\ell]}\overline{E}_{j_x}|=\binom{m-\ell}{2^s-1}$ is even.
\end{proof}

\begin{corollary}\label{cor_second_class_good}
 For any positive integer $s$ and $i\in [2^{s}]$,  $f_{\dot{\epsilon}_{i, 2^s}}(n)\geq (1-o(1))((2^s-1)! n)^{1\slash (2^s-1)}$  with $\dot{\epsilon}_{i, 2^s}\in \mathbb F_2^k$ of any length $k$.
\end{corollary}

\section{Three lemmas and their generalizations}\label{sec_preliminary}

In \cite{Johnston2023}, the authors provided three lemmas used as  fundamental tools for determining $f_{\alpha}(n)$. In this section, we generalize them for our purpose.

\begin{lemma}[Dual Lemma \cite{Johnston2023}]~\label{lemma_dual}
Let $\alpha\in \mathbb F_2^k\backslash\{\bm 0, \bm 1\}$. Then $f_{\alpha}(n)\sim f_{\alpha+\bm 1}(n)$.
\end{lemma}

\begin{lemma}[Trace Lemma \cite{Johnston2023}]~\label{lemma_trace}
Let $\alpha= (a_1, \ldots, a_k)\in \mathbb F_2^k$ and let $\beta= (a_{t+1}, \ldots, a_{t+r})\in \mathbb F_2^r$ where $t+r<k$, $f_{\alpha}(n)>k$ and $a_{t+1}\neq a_{t+2}$. Then
$$f_{\alpha}(n)\le f_{\beta}(n)+t.$$
\end{lemma}

\begin{lemma}[Partition Sum Lemma \cite{Johnston2023}]~\label{lemma_partition_sum}
Let $\alpha$ and $\beta$ be vectors in $\mathbb F_2^k$ and let $\gamma=\alpha+ \beta$. Then for any $r\in [n-1]$ where $f_\alpha(r), f_\beta(n-r)>k$ we have
$$f_{\gamma}(n)\ge \min\{f_\alpha(r), f_\beta(n-r)\}.$$
\end{lemma}

We remark that the essence of the partition sum lemma is a construction by concatenation. In fact, if $\mathcal F_1=(F_1, \ldots, F_{m_1})\subset 2^X$ and $\mathcal F_2=(F_1', \ldots, F_{m_2}')\subset 2^{X'}$ are intersecting families of patterns $\alpha$, $\beta \in \mathbb F_2^k$ based on two disjoint ground sets $X$ and $X'$, respectively, then the concatenation of  $\mathcal F_1$ and $\mathcal F_2$, i.e.,
$$\mathcal F=\{F_i\cup F_i': 1\le  i\le  \min\{m_1, m_2\} \}$$
is naturally an intersecting family of cardinality $\min\{m_1, m_2\} $ on {$X\cup X'$} with pattern $\alpha+\beta$. Every time we apply the partition sum lemma, we apply the concatenation construction to merge the existing constructions.

{Note that the trace lemma and the partition sum lemma require too many conditions. Since we mainly work on asymptotic cases, we will modify them into more applicable forms.} We utilize an immediate consequence from our constructions in Section~\ref{sec:two-fam}, which
ensures that for any given intersection pattern  $\alpha$, $f_{\alpha}(n)\to\infty$ as $n\to\infty$.

\begin{lemma}\label{theorem_lower_infty}
For any  $\alpha\in \mathbb F_2^k$, if  $t\in [k]$ is the maximum integer such that $a_t=1$, then $f_\alpha(n)=\Omega(n^{1\slash t})$.

\end{lemma}
\begin{proof}
Denote $P$ as the set of non-zero positions of $\alpha$. If $|P|=1$, i.e., $\alpha=\epsilon_t$, then by {Corollary~\ref{cor_first_class_good}, $f_\alpha(n)=\Omega(n^{ 1\slash t})$.} 

Otherwise, $\alpha$ is a sum of more than two patterns of $\epsilon_i$ with $i\in P$. By {Corollary~\ref{cor_first_class_good}},  $f_{\epsilon_i}=\Omega(n^{ 1\slash i})=\Omega(n^{ 1\slash t})$ since $i\leq t$ for each $i\in P$. By the concatenation construction, or equivalently, Lemma~\ref{lemma_partition_sum}, $f_\alpha(n)\ge \min_{i\in P}\{f_{\epsilon_i}(n\slash |P|)\}=\Omega(n^{ 1\slash t})$.
\end{proof}

By the lower bound in Lemma~\ref{theorem_lower_infty}, we can deduce the following  generalization of trace lemma in an asymptotic view. Here, we use the subscript $r$ in $\bm 1_r, \bm 0_r$ to indicate the length of these two vectors.

\begin{lemma}[Asymptotic Trace Lemma]\label{lemma_asym_trace}
For $\alpha= (a_1, \ldots, a_k)\in \mathbb F_2^k$ of some given length $k$ and its subsequence $\beta=(a_{t+1}, \ldots, a_{t+r})\in \mathbb F_2^r$ {with $t+r\leq k$ and $\beta\not\in\{\bm 1_r, \bm 0_r\}$, we have$f_\alpha(n)=O(f_\beta(n))$.}
\end{lemma}
\begin{proof}
By Lemma~\ref{theorem_lower_infty}, $f_\alpha(n)>k$ when $n$ is big enough. Let $\mathcal F$ be any $\alpha$-intersecting  family of size $f_\alpha(n)$. Fix $t$ distinct members $F_1, \ldots, F_t$ from $\mathcal F$ and set $T=\bigcap_{i\in [t]}F_i$. Then the family
$$\mathcal F_T:=\{F\cap T: F\in\mathcal F\backslash\{F_1, \ldots, F_t\}\}$$
is a $\beta$-intersecting family.

Let $s\in [r]$ be the smallest number such that $a_{t+1}\neq a_{t+s}$, which  exists by $\beta\not\in\{\bm 1_r, \bm 0_r\}$. We claim that there do not exist $s$ distinct members $F_{t+1}, \ldots, F_{t+s}\in \mathcal F\backslash\{F_1, \ldots, F_t\}$ such that $F_{t+i}\cap T$ are all equal for $i\in [s]$. Otherwise, we have
$$\bigcap_{i\in [t+s]}F_i=\bigcap_{i\in [s]}(F_{t+i}\cap T)=\bigcap_{i\in [t+1]}F_i,$$
which leads to a contradiction because $a_{t+1}\neq a_{t+s}$.

As a result, $f_{\beta}(n)\ge |\mathcal F_T|\ge (f_\alpha(n)-t)\slash (s-1)$, and hence {$f_\alpha(n)=O(f_\beta(n))$.}
\end{proof}

The following result is a generalization of the partition sum lemma, which shows that the size of the intersecting family for a combination pattern is at least as that for individuals.

\begin{lemma}[Lower Bound Lemma]\label{lemma_f_upper_bound}
Let $\alpha_1, \ldots, \alpha_t\in\mathbb F_2^k$  be linearly independent vectors for some given length $k$. If there exists a constant $c>0$ such that $f_{\alpha_i}(n)=\Omega(n^c)$ for each $i\in [t]$, then for any $\alpha\in \langle\alpha_1, \ldots, \alpha_t\rangle$, we have $f_\alpha(n)=\Omega(n^c)$.
\end{lemma}
\begin{proof}
If $\alpha=\bm 0$, a trivial lower bound shows $f_{\bm 0}(n)\ge 2^{\lfloor n\slash2\rfloor}=\Omega(n^c)$ for any constant $c>0$.

Otherwise, we can uniquely write $\alpha$ in a sum of some members $\alpha_i$. Without loss of generality, say,
$$\alpha=\alpha_1+\cdots+\alpha_s$$
for some $s\in [t]$. Since $s\leq t\leq k$, then by the partition sum lemma,
$$f_\alpha(n)\ge\min_{i\in [s]}\{f_{\alpha_i}(\frac n s)\}=\Omega(n^c).$$
\end{proof}

{
\begin{lemma}[Substitution Lemma]\label{lemma_substitude}
Let  $\alpha_1, \ldots, \alpha_t \in\mathbb F_2^k$ be linearly independent vectors satisfying each $f_{\alpha_i}(n)=\Omega(n^{c_1})$ for some $c_1>0$. Suppose $\beta_1, \beta_2$ are two vectors satisfying that  $\beta_1-\beta_2\in\langle \alpha_1, \ldots, \alpha_t\rangle$. If  $f_{\beta_1}(n)\sim cn^{c_2}$ for some constants $0<c_2<c_1$ and $c>0$, then $f_{\beta_2}(n)\sim cn^{c_2}$.

\end{lemma}
\begin{proof}
Since $\beta_1-\beta_2\in\langle \alpha_1, \ldots, \alpha_t\rangle$, we can write $\beta_1-\beta_2$ uniquely as the sum of some $\alpha_i$.  Without loss of generality, say, for some $s\in [t]$,
$$\beta_1-\beta_2=\alpha_1+\cdots+\alpha_s.$$
This is equivalent to
$$\beta_1=\beta_2+\alpha_1+\cdots+\alpha_s \text{ and }\beta_2=\beta_1+\alpha_1+\cdots+\alpha_s.$$
 Set $n'= n^{(c_2\slash c_1+ 1)\slash 2}$, which satisfies $n'=o(n)$ since $c_2< c_1$. By the partition sum lemma, we have
\begin{equation}\label{eq1}f_{\beta_1}(n)\ge\min\left\{f_{\beta_2}(n-sn'), \min_{i\in [s]}\{f_{\alpha_i}(n')\}\right\},
\end{equation} and
\begin{equation}\label{eq2}f_{\beta_2}(n)\ge\min\left\{f_{\beta_1}(n-sn'), \min_{i\in [s]}\{f_{\alpha_i}(n')\}\right\}.
\end{equation}

Since $f_{\alpha_i}(n)=\Omega(n^{c_1})$,
$f_{\alpha_i}(n')=\Omega(n^{(c_1+c_2)\slash 2})$ for each $i\in [s]$. By Eq.~(\ref{eq1}), $f_{\beta_2}(n-sn')< \min_{i\in [s]}\{f_{\alpha_i}(n')\}$ since otherwise $f_{\beta_1}(n)=\Omega(n^{(c_1+c_2)\slash 2})$, {which leads to a contradiction to the hypothesis that $f_{\beta_1}(n)\sim cn^{c_2}$.} So $f_{\beta_1}(n)\ge f_{\beta_2}(n-sn')$.

Since $n'=o(n)$ and $s\le k$, $n-sn'=n(1-o(1))$. So $f_{\beta_2}(n-sn')\le f_{\beta_1}(n)\sim cn^{c_2}$ implies that $f_{\beta_2}(n)\leq (1+o(1))cn^{c_2}$. Further,
$$f_{\beta_1}(n-sn')\sim c[(1-o(1))n]^{c_2}=(1-o(1))cn^{c_2}=o(f_{\alpha_i}(n')).$$
Then by Eq.~(\ref{eq2}), $f_{\beta_2}(n)\ge f_{\beta_1}(n-sn')=(1-o(1))cn^{c_2}$. This means $ f_{\beta_2}(n)\sim cn^{c_2}$.
\end{proof}
}

\section{{A proof of Theorem~\ref{theorem_extend_k_t_oddtown}}}\label{sec_generalizations}


In this section we give a  proof of Theorem~\ref{theorem_extend_k_t_oddtown}. First we prove the   upper bound in Theorem~\ref{theorem_extend_k_t_oddtown}.


\begin{lemma}\label{lemma_upper_bound}
Let $k\geq 2t$ be positive integers. {For any  $\alpha\in \mathbb F_2^k$ with $t$ being the maximum integer such that $a_t=1$,} we have $f_{\alpha}(n)\le (1+o(1))(t! n)^{1\slash t}$.

\end{lemma}
\begin{proof}
Let $\mathcal F\subseteq 2^{[n]}$ be an $\alpha$-intersecting family of size $f_\alpha(n)$. Let $\mathcal F_t$ be the set of all distinct $t$-wise intersections of $\mathcal F$, that is,
$$\mathcal F_t=\{F_1\cap F_2\cap \cdots \cap F_t: F_i\in \mathcal F, F_i\neq F_j, 1\leq i<j\leq t\}.$$

First, we claim that $|\mathcal F_t|=\binom {|\mathcal F|} t$. Otherwise, there are two different $t$-tuples $\{F_i, i\in [t]\}$ and $\{F_j', j\in [t]\}$ of distinct members in $\mathcal F$ such that $\bigcap_{i\in [t]}F_i= \bigcap_{j\in [t]}F_j'$. Since $\mathcal F$ is $\alpha$-intersecting and $a_t=1$,

$$|(\bigcap_{i\in [t]}F_i)\cap (\bigcap_{j\in [t]}F_j')|= |\bigcap_{i\in [t]}F_i|\equiv 1\mod 2.$$
But $(\bigcap_{i\in [t]}F_i)\cap (\bigcap_{j\in [t]}F_j')$ is the intersection of at least $t+1$ distinct members in $\mathcal F$, which contradicts to the fact that $a_i=0$ for all $i\geq t+1$. 

Then, since $\mathcal F$ is $\alpha$-intersecting with $a_t=1$ and $a_i=0$ for $t+1\leq i\leq 2t\leq k$, it is easy to check that $\mathcal F_t\subseteq 2^{[n]}$ is a $(1, 0)$-intersecting family, which means
$$|\mathcal F_t|\le f_{(1, 0)}(n) =n$$by the oddtown property.
So $\binom {f_\alpha(n)} t= |\mathcal F_t|\le n$ and $f_\alpha(n)\le (1+o(1))(t! n)^{1\slash t}$.
\end{proof}

We mention that the condition $k\geq 2t$ is necessary in the proof of Lemma~\ref{lemma_upper_bound}, otherwise it is not true to get an oddtown family. We recall Theorem~\ref{theorem_extend_k_t_oddtown} as follows.
\begin{thmbis}{theorem_extend_k_t_oddtown}
Let $k\geq 2t$ be positive integers. For any  $\alpha\in \mathbb F_2^k$ with $t$ being the maximum integer such that $a_t=1$, we have
$$f_\alpha(n)\sim (t! n)^{1\slash t}.$$
%
\end{thmbis}
\begin{proof}
We prove the theorem for the special pattern $\epsilon_t\in\mathbb F_2^{k}$ first.
By Corollary~\ref{cor_first_class_good}, $f_{\epsilon_t}(n)\ge (1-o(1))(t! n)^{1\slash t}$.
By Lemma~\ref{lemma_upper_bound}, $f_{\epsilon_t}(n)\le (1+o(1))(t! n)^{1\slash t}$. Hence, $f_{\epsilon_t}(n)\sim (t! n)^{1\slash t}$ when $k\geq 2t$.

For any pattern $\alpha\in \mathbb F_2^k$ with $t$ being the maximum integer such that $a_t=1$, we know that $\alpha-\epsilon_t\in \langle\epsilon_1, \ldots, \epsilon_{t-1}\rangle$. By  Corollary~\ref{cor_first_class_good},  $f_{\epsilon_i}(n)=\Omega (n^{1\slash (t-1)})$ for each $i\in [t-1]$.
Then applying the substitution lemma with $f_{\epsilon_t}(n)\sim (t! n)^{1\slash t}$, we have $f_\alpha(n) \sim (t! n)^{1\slash t}$.
\end{proof}

\section{A Proof of Theorem~\ref{theorem_main}}\label{sec_main_result}

In this section, we give a proof of  Theorem~\ref{theorem_main}. We first handle a special case when the length $k$ is a power of $2$.

\begin{theorem}\label{theorem_pre_main}
Let $k=2^{\ell+1}$ for some positive integer $\ell$. For any pattern $\alpha=(a_1, \ldots, a_k)\in\mathbb F_2^k$, the followings hold.
\begin{itemize}
\item[(1)] If $a_{2^\ell}\ne a_{2^{\ell+1}}$, $f_\alpha(n)\sim (2^\ell! n)^{1\slash 2^\ell}$.
\item[(2)] Otherwise, $f_\alpha(n)=\Omega(n^{1\slash(2^\ell-1)})$.
\end{itemize}
\end{theorem}
\begin{proof}
Consider the following $k$ patterns of length $k$:
$$\text{$\epsilon_i$ for $i\in [2^\ell]$, $\dot{\epsilon}_{j, 2^\ell}$ for $j\in [2^\ell-1]$, and the all-one vector $\bm 1$.}$$
It is easy to check that those $k$ patterns exactly form a basis of $\mathbb F_2^k$.


By  Corollary~\ref{cor_first_class_good}, $f_{\epsilon_i}(n)=\Omega(n^{1\slash i})=\Omega(n^{1\slash (2^\ell-1)})$ for any $i\in [2^\ell-1]$. By  Corollary~\ref{cor_second_class_good}, $f_{\dot{\epsilon}_{j, 2^\ell}}(n)=\Omega(n^{1\slash (2^\ell-1)})$ for any $j\in [2^\ell-1]$. Further, it is well known that $f_{\bm 1}(n)=\Omega(2^{n\slash 2})=\Omega(n^{1\slash (2^\ell-1)})$. By Theorem~\ref{theorem_extend_k_t_oddtown}, we have $f_{\epsilon_{2^\ell}}(n)\sim (2^\ell! n)^{1\slash 2^\ell}$.

For any $\alpha\in \mathbb F_2^k$ satisfying $a_{2^\ell}\ne a_{2^{\ell+1}}$, it is easy to check that
$$\alpha-\epsilon_{2^\ell}\in\langle\epsilon_1, \ldots, \epsilon_{2^\ell-1}, \dot{\epsilon}_{1, 2^\ell}, \ldots, \dot{\epsilon}_{2^\ell-1, 2^\ell}, \bm 1 \rangle.$$
Then by the substitution lemma, $f_\alpha(n)\sim f_{\epsilon_{2^\ell}}(n)\sim (2^\ell! n)^{1\slash 2^\ell}$.

Otherwise, $a_{2^\ell}= a_{2^{\ell+1}}$. Then
$$\alpha\in\langle\epsilon_1, \ldots, \epsilon_{2^\ell-1}, \dot{\epsilon}_{1, 2^\ell}, \ldots, \dot{\epsilon}_{2^\ell-1, 2^\ell}, \bm 1 \rangle.$$
From the lower bound lemma, i.e., Lemma~\ref{lemma_f_upper_bound}, $f_\alpha(n)=\Omega(n^{1\slash(2^\ell-1)})$.
\end{proof}

{By  Theorem~\ref{theorem_pre_main} (1), it is immediate to have the following consequence, which confirms Conjecture~\ref{conjecture_1}  when $\mathcal A_1=\cdots =\mathcal A_k$ and $k$ is a power of $2$.

\begin{corollary}\label{cor-conje}
  Let $k=2^{\ell+1}$ for some positive integer $\ell$ and let $t\in [2^\ell,k-1]$. For  $\alpha=(1,1,\ldots, 1, 0,\ldots, 0)\in \mathbb F_2^k$ with the first $t$ coordinates being one,
 we have
$$f_\alpha(n)\sim (2^\ell! n)^{1\slash 2^\ell}.$$
\end{corollary}}


Now we prove Theorem~\ref{theorem_main}.

\begin{thmbis}{theorem_main}
For any positive integers $\ell$ and $k\geq 2$, the number of patterns $\alpha\in\mathbb F_2^k$ such that $f_{\alpha}(n)=\Omega(n^{1\slash (2^\ell -1)})$ is no more than $2^{2^{\ell+1}-1}$. For all other patterns not satisfying this property, $f_{\alpha}(n)=O(n^{1\slash 2^\ell})$.
\end{thmbis}
\begin{proof}
If $k\le 2^{\ell+1}-1$, there is nothing to prove because the number of distinct length $k$ patterns does not exceed $2^{2^{\ell+1}-1}$. If {$k= 2^{\ell+1}$}, it has been proved by Theorem~\ref{theorem_pre_main}. So we only need to consider the case when {$k> 2^{\ell+1}$}.  Let $B$ be the set of all patterns $\alpha\in \mathbb F_2^k$  satisfying $f_{\alpha}(n)=O(n^{1\slash 2^\ell})$, and let $G= \mathbb F_2^k\backslash B$.


First, we show that $|G|\le 2^{2^{\ell+1}-1}$. For any pattern $\alpha=(a_1, \ldots, a_k)\in G$, we claim that $a_j= a_{j+2^\ell}$ for all $j\in [2^\ell, k-2^\ell]$. Otherwise, if there exists some $j\in [2^\ell, k-2^\ell]$ such that $a_j\neq a_{j+2^\ell}$, we consider the substring $\beta=(a_{j-2^\ell+1}, a_{j-2^\ell+2}, \ldots, a_{j+2^\ell})$ of length $2^{\ell+1}$. We know the $2^\ell$-th position and the $2^{\ell+1}$-th position of $\beta$ have different values (since $a_j\neq a_{j+2^\ell}$), so $f_\beta(n)=\Theta(n^{1\slash 2^\ell})$ by Theorem~\ref{theorem_pre_main} (1). Then by the asymptotic trace lemma, $f_\alpha(n)=O(n^{1\slash 2^\ell})$, which implies $\alpha\in B$ {and leads to a contradiction since $\alpha\in G$.}
As a result, for any $\alpha \in G$, the first $2^{\ell+1}-1$ values in $\alpha$ are enough to define the whole $\alpha$, and the rest part of $\alpha$ repeats the substring $(a_{2^\ell}, a_{2^\ell+1}, \ldots, a_{2^{\ell+1}-1})$ over and over again. So $|G|\le 2^{2^{\ell+1}-1}$.

Next, we show that any $\alpha$ having the above recursive form of period $2^\ell$  satisfies $f_{\alpha}(n)=\Omega(n^{1\slash (2^\ell -1)})$, and hence in $G$.
As in the  proof of Theorem~\ref{theorem_pre_main}, we consider the following $2^{\ell+1}-1$ special independent patterns of length $k$:
$$\text{$\epsilon_i$ for $i\in [2^\ell-1]$ and $\dot{\epsilon}_{j, 2^\ell}$ for $j\in [2^\ell]$.}$$
It is easy to check that the considered $\alpha$  satisfies
$$\alpha\in\langle\epsilon_1, \ldots, \epsilon_{2^\ell-1}, \dot{\epsilon}_{1, 2^\ell}, \ldots, \dot{\epsilon}_{2^\ell, 2^\ell}\rangle.$$
In fact, if we define three binary vectors from $\alpha$ of length $2^\ell$: $\alpha_1=(a_1, \ldots, a_{2^\ell})$, $\alpha_2= (a_{2^\ell+1}, \ldots, a_{2^{\ell+1}})$, and $\alpha_3=\alpha_1-\alpha_2$, and let $P, P'\subset [2^\ell]$ be the sets of nonzero positions of $\alpha_3$ and $\alpha_2$, respectively, we have
$$\alpha=\sum_{i\in P}\epsilon_i+\sum_{j\in P'}\dot{\epsilon}_{j, 2^\ell}.$$
By Corollary~\ref{cor_first_class_good}, Corollary~\ref{cor_second_class_good},
  and the lower bound lemma, Lemma~\ref{lemma_f_upper_bound}, we have $f_{\alpha}(n)=\Omega(n^{1\slash (2^\ell -1)})$.
\end{proof}

{
\begin{remark}\label{rmk_equal}
From the proof, we see that when $k\geq 2^{\ell+1}$, the number of  patterns $\alpha\in\mathbb F_2^k$ such that $f_{\alpha}(n)=\Omega(n^{1\slash (2^\ell -1)})$ is exactly $2^{2^{\ell+1}-1}$.
\end{remark}}

\section{conclusion}\label{sec_conclude}

We considered the asymptotic behavior of  the maximum size of an $\alpha$-intersecting family with $\alpha\in\mathbb F_2^k$, which has restrictions on all $\ell$-wise intersections for $\ell\leq k$. This is a much more general problem originated from the classical oddtown and eventown problems.
We first gave two basic constructions of $\mathcal F_{s, i}(n)$ and $\overline{\mathcal F}_{s, i}(n)$, which have good intersection patterns under certain choices of $s$ and $i$. They are used as building blocks in the concatenation  method to obtain families of nearly optimal size.
Combining with several generalizations of the trace lemma and the partition sum lemma, we determine the asymptotic values of $f_{\alpha}(n)$ for many patterns $\alpha$ satisfying certain restrictions (Theorem~\ref{theorem_extend_k_t_oddtown}), and upper bound the values of $f_{\alpha}(n)$ for almost all patterns $\alpha$ with a small constant exponent on $n$ (Theorem~\ref{theorem_main}). Our results completely answer an open question  in \cite{Johnston2023} and partially confirm a conjecture in \cite{o2022note}.


We leave several open problems below for further study.

\begin{itemize}
\item[(1)] We have answered Question~\ref{question_1} by saying that $C=2^7$ is enough. But we have also pointed out that $2^7$ might not be the  best.
\begin{problem}
Find the best possible $C$ for Question~\ref{question_1}.
\end{problem}

\item[(2)] There are still many patterns $\alpha$ for which the asymptotic behaviors  of $f_\alpha(n)$ are not known, even for very short $\alpha$. For example, we know when $\alpha=(0, 0, 0, 0, 1)$, $O(n^{1\slash 2})=f_\alpha(n)= \Omega(n^{1\slash 3})$, but we do not know the exact exponent. One possible way is to try to explore more properties of the two basic families $\mathcal F_{s, i}(n)$ and $\overline{\mathcal F}_{s, i}(n)$ for more parameters, since in this work, we only used a very little fraction of parameters (see Lemma~\ref{lemma_first_class_good} and Lemma~\ref{lemma_second_class_good}).

\begin{problem}
For any given length $k$, find asymptotic results of $f_\alpha(n)$ for all $\alpha\in\mathbb F_2^k$.
\end{problem}

%
\item[(3)]
We have smoothly generalized Theorem~\ref{theorem_extend_k_t_oddtown} and Theorem~\ref{theorem_main} to the modulo $p$ setting when $\alpha\in \mathbb F_p^k$, which are Theorem~\ref{theorem_k_reachable_2} and Theorem~\ref{theorem_main_2}, respectively. Theorem~\ref{theorem_extend_k_t_oddtown} can also be extended to $\alpha\in\left\{\mathbb F_p\cup\{\star\}\right\}^k$ in the modulo $p$ setting, which is Theorem~\ref{theorem_k_reachable_2_star}. Here remains the last piece.
\begin{problem}
Find the analog of Theorem~\ref{theorem_main} in the modulo $p$ case when $\alpha\in \left\{\mathbb F_p\cup\{\star\}\right\}^k$.
\end{problem}
\end{itemize}



\appendices
\section{Modulo $p$ intersecting family of pattern $\alpha$}\label{sec_appendix}
For any nonnegative integer $s$ and { prime} $p$, denote $\varphi_p(s)$ as the number of $p$-factors in $s!$.  Under the modulo $p$ setting, the two families defined in Construction~\ref{construction_building_block} can be extended by just replacing $2$ by $p$ and $\varphi_p(s)$ by $\varphi_p(s)$. We denote the two new families by $\mathcal F_{s, i}(p, n)$ and
$\overline{\mathcal F}_{s, i}(p, n)$ respectively. For any $i \in [0, p^{\varphi_p(s)+1}-1]$, the size $m$ of $\mathcal F_{(s, i)}(p, n)$ and $\overline{\mathcal F}_{(s, i)}(p, n)$ also satisfies $m \sim (s! n)^{1\slash s}$ when $n$ goes to infinity.

Similar to Lemma~\ref{lemma_first_class_good} and Lemma~\ref{lemma_second_class_good}, we have the following two lemmas. The proof of Lemma~\ref{lemma_first_class_good_2} is almost the same as Lemma~\ref{lemma_first_class_good} thus omitted.

\begin{lemma}\label{lemma_first_class_good_2}
For any prime $p$, positive integers $s$ and $n$, if $|\mathcal F_{s, s-1}(p, n)|=m\ge s+1$, then $\mathcal F_{s, s-1}(p, n)$ is $\epsilon_s$-intersecting modulo $p$ with $\epsilon_s\in \mathbb F_p^m$. Hence $f_{\epsilon_s}(p,n)\geq (1-o(1))(s! n)^{1\slash s}$  with $\epsilon_s\in \mathbb F_p^k$.
\end{lemma}
%

\begin{lemma}\label{lemma_second_class_good_2}
For any prime $p$,  positive integer $s$ and $i\in [0, p^{s}-1]$, if $|\overline{\mathcal F}_{p^s-1, i}(p, n)|=m\geq p^s$, then it is
$\dot{\epsilon}_{i+1, p^s}$-intersecting modulo $p$ with $\dot{\epsilon}_{i+1, p^s}\in\mathbb F_p^m$. Hence $f_{\dot{\epsilon}_{i, p^s}}(p,n)\geq (1-o(1))((p^s-1)! n)^{1\slash (p^s-1)}$  with $\dot{\epsilon}_{i, p^s}\in \mathbb F_p^k$ for any $i\in [p^{s}]$.
\end{lemma}
\begin{proof}
For any $\ell$ distinct members $\overline{E}_{j_1}, \ldots \overline{E}_{j_\ell}$ in $\overline{\mathcal F}_{p^s-1, i}(p, n)$, it has been calculated that $|\bigcap_{x\in[\ell]}\overline{E}_{j_x}| = \binom{m-\ell}{p^s-1}$. 
Recall that $m\equiv i\mod p^{\varphi_p(p^s-1)+1}$ and hence $m\equiv i\mod p^{s}$.

If $\ell\not\equiv i+1\mod p^s$, the argument is the same as that in Lemma~\ref{lemma_second_class_good}, which gives $\binom{m-s}{p^\ell-1}\equiv 0\mod p$.


Consider the case $\ell\equiv i+1\mod p^s$. First, we claim that for any $j\in [p^s-1]$ and $m'>j$ with $m'\equiv j-1 \mod p^s$, we have $\binom {m'} j\equiv 0\mod p$.  Write $\binom {m'} j$ as the form of $\left(\prod_{x\in [j]}(m'-j+x)\right)\slash j!$. Since $m'\equiv j-1\mod p^s$, $m'-j+x$ has the same number of $p$-factors as $x-1$ when $x\in [2, j]$, while $(m'-j+1)$ has more $p$-factors than $j$. Hence $\binom {m'} j\equiv 0\mod p$.

By repeatedly use the equation $\binom N k=\binom {N-1} k+\binom {N-1}{k-1}$ for any integers $N\ge k\ge 1$, we have
$$\binom{m-\ell}{p^s-1}=\binom{m-\ell-1}{p^s-1}+ \binom {m-\ell-1}{p^s-2}=\cdots =\sum_{j=1}^{p^s-1}\binom{m-\ell-j}{p^s-j}+\binom {m-\ell-p^s+1} 0.$$
Since $m\equiv i\mod p^s$ and $\ell\equiv i+1\mod p^s$, for any $j\in [p^s-1]$, we have $m-\ell-j\equiv p^s- j-1\mod p^s$, and hence $\binom{m-\ell-j}{p^s-j}\equiv 0\mod p^s$. As a result, $\binom{m-\ell}{p^s-1}\equiv 1\mod p^s.$
\end{proof}

The followings are the modulo $p$ extensions of asymptotic trace lemma, lower bound lemma and substitution lemma. As the proving process goes much the same as in the modulo 2 setting, we leave the detailed proofs to readers.

\begin{lemma}[Asymptotic Trace Lemma]\label{lemma_asym_trace_2}
For $\alpha= (a_1, \ldots, a_k)\in \mathbb F_p^k$ of some given length $k$ and its subsequence $\beta=(a_{t+1}, \ldots, a_{t+r})\in \mathbb F_p^r$, as long as $a_{t+1}, \ldots, a_{t+r}$ are not all the same, $f_\alpha(p, n)=O(f_\beta(p, n))$.
\end{lemma}

\begin{lemma}[Lower Bound Lemma]\label{lemma_f_upper_bound_2}
Let $\alpha_1, \ldots, \alpha_t\in\mathbb F_p^k$ be linearly independent vectors for some given length $k$. If there exists a constant $c>0$ such that $f_{\alpha_i}(p, n)=\Omega(n^c)$ for any $i\in [t]$, then for any $\alpha\in \langle\alpha_1, \ldots, \alpha_t\rangle$, we have $f_\alpha(p, n)=\Omega(n^c)$.
\end{lemma}

\begin{lemma}[Substitution Lemma]\label{lemma_substitude_2}
Let $t\le k$ be positive integers and $\alpha_1, \ldots, \alpha_t\in\mathbb F_p^k$ be linearly independent vectors satisfying $f_{\alpha_i}(p, n)=\Omega(n^{c_1})$ for some $c_1>0$. Suppose $\beta_1, \beta_2 \in\mathbb F_p^k$ are two vectors satisfying
$\beta_1-\beta_2\in \langle\alpha_1, \ldots, \alpha_t\rangle$. Then for some constants $0<c_2<c_1$ and $c>0$,
\begin{itemize}
\item[(1)] $f_{\beta_2}(p, n)= \Theta(n^{c_2})$ if $f_{\beta_1}(p, n) =\Theta(n^{c_2})$; and
\item[(2)] $f_{\beta_2}(p, n)\sim cn^{c_2}$ if $f_{\beta_1}(p, n)\sim cn^{c_2}$.
\end{itemize}
\end{lemma}

Next we come to estimating the upper bound for  ceratin $\alpha\in\mathbb F_p^k$. We will need the following celebrating Frankl-Wilson Theorem.
\begin{theorem}[\hspace{-0.01em}\cite{Frankl1981}]\label{theorem_franke_wilson}
Let $p$ be a prime, $L\subset \mathbb F_p$ be a set of residues modulo $p$, and $\mathcal F\subset 2^{[n]}$ be a family such that $|A|\notin L$ for any $A\in \mathcal F$, but $|A\cap B|\in L$ for any distinct $A, B\in\mathcal F$. Then
$$|\mathcal F|\le \sum_{s=0}^{|L|}\binom n s.$$
\end{theorem}

\begin{lemma}\label{lemma_upper_2}
For any $\alpha=(a_1, \ldots, a_k)\in\mathbb F_p^k$ for some prime $p$ and length $k>0$ such that for some positive integer $t\le k\slash 2$, $a_t\ne a_{t+1}=\cdots =a_k$, we have
$$f_\alpha(p, n)\le (1+o(1))(t! n)^{1\slash t}.$$
\end{lemma}
\begin{proof}
Suppose $\mathcal F\subset 2^{[n]}$ is an $\alpha$-intersecting family modulo $p$ of size $f_\alpha(p, n)$. Then consider $\mathcal F_t$, which is the collection of all distinct $t$-wise intersections of $\mathcal F$.

As the same analysis as in Lemma~\ref{lemma_upper_bound}, since $a_t\ne a_{t+1}$, $|\mathcal F_t|=\binom{|F|} t$, and $\mathcal F_t$ is an $(a_t, a_{t+1})$-intersecting family modulo $p$.
Applying Theorem~\ref{theorem_franke_wilson} with $L=\{a_k\}$, we have $|\mathcal F_t|\le n+1$. Hence
$$f_\alpha(p, n)=|\mathcal F|\le (1+o(1))(t! n)^{1\slash t}.$$
\end{proof}

The dual lemma can also be extended to the modulo $p$ setting as follows. A new name is given because $\alpha+\bm 1$ is no longer the dual of $\alpha$ when $p> 2$.

\begin{lemma}[Shift Lemma]\label{lemma_shift}
For any $\alpha=(a_1, \ldots, a_k)\in \mathbb F_p^k$, $f_\alpha(p, n)=\Theta(f_{\alpha+\bm 1}(p, n))$. If further $a_1, \ldots, a_k$ are not all equal, $f_\alpha(p, n) \sim f_{\alpha+\bm 1}(p, n)$.
\end{lemma}
\begin{proof}
For any family $\mathcal F\subset 2^{[n]}$  which is $\alpha$-intersecting modulo $p$, the new family $\big\{F\cup\{n+1\}: F\in \mathcal F\big\}\subset 2^{[n+1]}$ is $(\alpha+\bm 1)$-intersecting modulo $p$. Hence
$$f_\alpha(p, n)\le f_{\alpha+\bm 1}(p, n+1)\le f_{\alpha+2\cdot\bm 1}(p, n+2)\le\cdots\le f_{\alpha+ p\cdot\bm 1}(p, n+p)= f_\alpha(p, n+p).$$
Here $i\cdot \bm 1$ is the all-$i$ vector.
By the definition, $f_\alpha(p,n)$ grows with $n$. If $a_1=\cdots =a_k$, there exists some $j\in [p]$ such that $\alpha+j\cdot\bm 1=\bm 0$, and hence $f_\alpha(p, n)=\Theta(f_{\bm 0}(p, n))= \Theta(f_{\alpha+\bm 1}(p, n))$. If $a_1, \ldots, a_k$ are not all equivalent, by Lemma~\ref{lemma_upper_2} with $t=1$, $k=2$ and asymptotic trace lemma, $f_\alpha(p, n)=O(n)$. Hence $f_\alpha(p, n)\sim f_\alpha(p, n+p)$ and $f_\alpha(p, n)\sim f_{\alpha+\bm 1}(p, n)$.
\end{proof}

The following is an analog of Theorem~\ref{theorem_extend_k_t_oddtown}. The form becomes a little more complicated because a nonzero value of any position in $\alpha$ has more than one choice. The proving strategy is much the same.
\begin{theorem}\label{theorem_k_reachable_2}
Let $p$ be a prime.
 For any positive $t$ and  pattern $\alpha=(a_1, \ldots, a_k)\in\mathbb F_p^{k}$ with $k\ge 2t$, if
  $a_t\ne a_{t+1}=\cdots =a_k$,
we have $f_\alpha(p, n) =\Theta(n^{1\slash t})$.
Moreover, if $a_t-a_{t+1}=1$, $f_\alpha(p, n) \sim (t! n)^{1\slash t}$.
\end{theorem}
\begin{proof}
We only need to prove the case when $a_{t+1}=0$. Otherwise we can use the shift lemma to reduce it to the $a_{t+1}=0$ case.

As in the proof of Theorem~\ref{theorem_extend_k_t_oddtown}, consider $\epsilon_i\in\mathbb F_p^k$ for any $i\in [t]$. Then by Lemma~\ref{lemma_first_class_good_2}, $f_{\epsilon_i}(p, n)=\Omega(n^{1\slash i})$ for any $i\in [t]$, and specially $f_{\epsilon_t}(p, n)\ge (1-o(1))(t! n)^{1\slash t}$.
Since $k\ge 2t$, by Lemma~\ref{lemma_upper_2}, $f_{\epsilon_t}(p, n) \sim (t! n)^{1\slash t}$.

If $a_t=1$ for $\alpha$, it is easy to check that
$$\alpha-\epsilon_t\in \langle\epsilon_1, \ldots, \epsilon_{k-1}\rangle.$$
By Lemma~\ref{lemma_substitude_2} (2),  $f_\alpha(p, n)\sim f_{\epsilon_t}(p, n) \sim (t! n)^{1\slash t}$.

Otherwise, let $\beta=a_t\cdot \epsilon_t$, i.e., the $t$-th position of $\beta$ has value $a_t$ and other positions have common value $0$. Then by partition sum lemma, $f_{\beta}(p, n)\ge f_{\epsilon_t}(p, n\slash a_t)=\Omega(n^{1\slash t})$, and by Lemma~\ref{lemma_upper_2}, $f_{\beta}(p, n)= \Theta(n^{1\slash t})$.

Finally, it is easy to check that
$$\alpha-\beta\in \langle\epsilon_1, \ldots, \epsilon_{k-1}\rangle.$$
From Lemma~\ref{lemma_substitude_2} (1), $f_\alpha(p, n) =\Theta(f_\beta(p, n))= \Theta(n^{1\slash t})$.
\end{proof}

Then comes the analogs of Theorem~\ref{theorem_pre_main} and Theorem~\ref{theorem_main}. The proof strategies are the same and hence we omit the proofs.
\begin{theorem}\label{theorem_pre_main_2}
Let $p$ be a prime, and let $k=2\cdot p^\ell$ for some positive integer $\ell$. For any pattern $\alpha=(a_1, \ldots, a_k)\in\mathbb F_p^k$, the followings hold.
\begin{itemize}
\item[(1)] If $a_{p^\ell}\ne a_{2\cdot p^\ell}$, $f_\alpha(p, n)= \Theta(n^{1\slash p^\ell})$.
\item[(2)] Otherwise, $f_\alpha(p, n)=\Omega(n^{1\slash(p^\ell-1)})$.
\end{itemize}
\end{theorem}

\begin{theorem}\label{theorem_main_2}
Let $p$ be a prime. For any positive integers $k$ and $\ell$, the number of patterns $\alpha\in\mathbb F_p^k$ such that $f_{\alpha}(p, n)=\Omega(n^{1\slash (p^\ell -1)})$ is no more than $p^{2\cdot p^\ell -1}$. For all other patterns not satisfying this property, $f_{\alpha}(p, n)=O(n^{1\slash p^\ell})$.
\end{theorem}

Similarly, the number of patterns $\alpha\in\mathbb F_p^k$ such that $f_{\alpha}(p, n)=\Omega(n^{1\slash (p^\ell -1)})$ is exactly $p^{2\cdot p^\ell -1}$ when $k\ge 2\cdot p^\ell$.

Finally, we consider the modulo $p$ result allowing $\alpha_i$ to be $\star$. We still have a similar upper bound as Lemma~\ref{lemma_upper_2}.

\begin{lemma}\label{lemma_upper_2_star}
For any $\alpha=(a_1, \ldots, a_k)\in (\mathbb F_p\cup\{\star\})^k$ for some prime $p$ and positive integer $t\le k\slash 2$ such that $a_t\ne 0$ and $ a_{t+1}=\cdots = a_k=0$, we have
$$f_\alpha(p, n)\le (1+o(1))(t! n)^{1\slash t}.$$
\end{lemma}
The proof of Lemma~\ref{lemma_upper_2_star} is the same as Lemma~\ref{lemma_upper_2}, and the point is to consider the $(\star, 0)$-intersecting family $\mathcal F_t$ after given the extremal family $\mathcal F$, and then to apply Theorem~\ref{theorem_franke_wilson}. From Lemma~\ref{lemma_upper_2_star} and Theorem~\ref{theorem_k_reachable_2}, we have the following simple corollary, as an extension of Theorem~\ref{theorem_extend_k_t_oddtown}.

{
\begin{theorem}\label{theorem_k_reachable_2_star}
Let $p$ be a prime. For any  positive $\ell$ and any pattern $\alpha=(a_1, \ldots, a_k)\in \left\{\mathbb F_p\cup\{\star\}\right\}^k$ with $k\ge 2\ell$ such that $a_\ell\ne 0$ and $a_{\ell+1}=\cdots= a_k=0$,
we have $f_\alpha(p, n) = \Theta( n^{1\slash \ell})$. Moreover, if $a_\ell=1$ or $a_\ell=\star$, $f_\alpha(p, n) \sim (\ell! n)^{1\slash \ell}$.
\end{theorem}
}



\end{document}